\documentclass[12pt,twoside,a4paper]{article}
\usepackage{euscript,a4,times}
\usepackage[cp866]{inputenc}
\usepackage[english]{babel}
\usepackage{makeidx}
\usepackage{color}
\usepackage{latexsym,amsfonts,amssymb,amsmath,longtable,amsthm}
\usepackage{ulem}
\usepackage[table]{xcolor}

\catcode`@=11
\newbox\tr@tto
\setbox\tr@tto=\hbox{{\count0=0\dimen0=-
,9pt\dimen1=1,1pt\loop\ifnum\count0<11
\advance \count0 by1 \vrule width.51pt height\dimen1
     depth\dimen0\kern-0.17pt\advance\dimen0 by-
0.05pt\advance\dimen1
     by0.1pt\repeat \loop\ifnum\count0<21\advance \count0 by1
\vrule
     width.6pt height\dimen1 depth\dimen0\kern-0.2pt
\advance\dimen0
     by-0.1pt\advance\dimen1 by 0.05pt\repeat}}
\def\medint{\displaystyle\copy\tr@tto\kern-10.4pt\int}
\catcode`@=12

\input amssym.def

\def\Xint#1{\mathchoice
   {\XXint\displaystyle\textstyle{#1}}%
   {\XXint\textstyle\scriptstyle{#1}}%
   {\XXint\scriptstyle\scriptscriptstyle{#1}}%
   {\XXint\scriptscriptstyle\scriptscriptstyle{#1}}%
   \!\int}
\def\XXint#1#2#3{{\setbox0=\hbox{$#1{#2#3}{\int}$}
     \vcenter{\hbox{$#2#3$}}\kern-.5\wd0}}

\def\dashint{\Xint-}
\newcommand{\R}{{\mathbb R}}

\newcommand{\Le}{{\mathcal L}}
\newcommand{\E}{\mathop{\rm Ent}}
\newcommand{\M}{{\mathcal  M}}
\renewcommand{\H}{{\mathcal H}}
\newcommand{\rank}{{\mathrm{rank}}}

\newcommand{\supp}{\mathop{\rm supp}}

\def\ve{\varepsilon}
\newcommand{\loc}{{\rm loc}}

\newcommand{\esssup}{\mathop{\rm ess\,sup}}

\newcommand{\diam}{\mathop{\rm diam}}

\newcommand{\dist}{\mathop{\rm dist}}

\newcommand{\LL}{\mathrm{L}}
\newcommand{\WW}{\mathrm{W}}

\newcommand{\CC}{\mathrm{C}}
\newcommand{\BV}{\mathrm{BV}}

\newtheorem{lem}{Lemma}[section]
\newtheorem{ttt}[lem]{Theorem}
\newtheorem{cor}[lem]{Corollary}
\theoremstyle{definition}
\newtheorem{rem}[lem]{Remark}

\newtheorem{df}[lem]{Definition}

\newcommand{\dd}{{\rm d}}

\title{ON THE MORSE--SARD PROPERTY AND LEVEL SETS OF $\WW^{n,1}$
SOBOLEV FUNCTIONS ON $\R^n$}
\author{Jean Bourgain, Mikhail V.~Korobkov\footnote{The author was
supported by the Russian Foundation for Basic Research (project
No.~12-01-00390-a) and by the Integration Project SB-FEB RAS
(project No.~56). } \, and Jan Kristensen\footnote{Work supported
by the EPSRC Science and Innovation award to the Oxford Centre for
Nonlinear PDE (EP/E035027/1).}}

\date{}

\begin{document}

\maketitle

%
%

\begin{abstract}
We establish Luzin $N$ and Morse--Sard properties for functions
from the Sobolev space $\WW^{n,1}(\R^n)$. Using these results we
prove that almost all level sets are finite disjoint unions of
$\CC^1$-smooth compact manifolds of dimension $n-1$. These results
remain valid also within the larger space of functions of bounded
variation $\BV_n(\R^n)$. For the proofs we establish and use some
new results on Luzin--type approximation of Sobolev and
$\BV$--functions by $\CC^k$--functions, where the exceptional sets
have small Hausdorff content.

\end{abstract}

\section*{Introduction}

The starting point of the paper is the following classical result
(see also \cite{Fed} for more general expositions):
\bigskip

\noindent {\bf Theorem (Morse-Sard, 1942, \cite{Mo}, \cite{S})}.
{\sl Let
 $f \colon \R^n\to\R^m$ be a
$\CC^k$--smooth mapping with $k\ge\max(n-m+1,1)$. Then
\begin{equation}\label{t1}
\Le^m(f(Z_f))=0, \end{equation} where  $\Le^m$ denotes the
$m$-dimensional Lebesgue measure and $Z_f$ denotes the set of
critical points of $f$:\ \ $Z_f=\{x\in\R^n\,:\,\rank \nabla
f(x)<m\}$.}
\bigskip

The order of smoothness in the assumptions of this theorem is
sharp on the scale $\CC^j$ (see, e.g., \cite{Wh}). However, some
analogs of the Morse--Sard theorem remain valid for functions
lacking the required smoothness in the classical theorem.
Although~(\ref{t1}) may be no longer valid then,
Dubovitski\u{\i}~\cite{Du} obtained some results on the structure
of level sets in the case of reduced smoothness (also
see~\cite{B}).

Another direction of the research was the generalization of the
Morse--Sard theorem to functions in more refined scales of spaces,
and especially in H\"{o}lder and Sobolev spaces (for example,
see~\cite{Bates,B,DeP,Fig,Nor}). In particular, De~Pascale
(\cite{DeP}, see also \cite{Fig}) proved that~(\ref{t1}) holds
when $f\in \WW^{k,p}(\R^n,\R^m)$ with $p>n$, $k\ge\max(n-m+1,2)$.
Note that in this case $v$ is $\CC^1$--smooth by virtue of the
Sobolev imbedding theorem, and so the critical set is defined as
usual.

For a historical review for the plane case $n=2,m=1$ see for
instance \cite{BKK}. We mention only the paper~\cite{PZ} where it
was proved that~(\ref{t1}) holds for Lipschitz functions $f$ of
class $\BV_{2}(\R^2)$, where $\BV_{2}(\R^2)$ is the space of
functions $f\in \LL^{1}(\R^2)$ such that all its partial
(distributional) derivatives of the second order are $\R$-valued
Radon measures on $\R^2$.

In this paper we  consider the case of $\R$-valued Sobolev
functions $v\in \WW^{n,1}(\R^n)$. It is known (see, e.g.,
\cite{Dor}) that such a function  admits a continuous
representative which is \linebreak(Fr\'{e}chet--)differentiable
$\H^1$--almost everywhere. The critical set $Z_v$ is defined as
the set of points~$x$, where $v$ is differentiable with total
(Fr\'{e}chet--)differential $v'(x)=0$. As our  main result we
prove that $\Le^1(v(Z_v))=0$ (see Theorem~\ref{MS}).

Also we show that for any $v\in \WW^{n,1}(\R^n)$ and $\varepsilon
>0$ there exists $\delta >0$ such that for all subsets
$E\subset\R^n$ with $\H_\infty^1(E) < \delta$ we have
$\Le^1(v(E))< \varepsilon$, where $\H^1_\infty$ is the Hausdorff
content. In particular, it follows that $\Le^1(v(E))=0$ whenever
$\H^1(E)=0$ (see Theorem~\ref{th3.3}). So the image of the
exceptional ``bad'' set, where the differential is not defined,
has zero Lebesgue measure. This ties nicely with our definition of
the critical set and our version of the Morse--Sard result.

Finally, using these results we prove that almost all level sets
of $\WW^{n,1}$--functions defined on $\R^n$, are finite disjoint
unions of $\CC^1$--smooth compact manifolds of dimension $n-1$
without boundary (see Theorem~\ref{Th2.1}\,).

The proof of the last result relies in turn on new Luzin--type
approximation results for $\WW^{l,1}$ Sobolev functions by
$\CC^k$--functions, $k \leq l$, where the exceptional sets are of
small Hausdorff content (see Theorem~\ref{Th_ap}\,). The $\LL^p$
analogs of such results are well-known when $p>1$, see, e.g.,
\cite{Boj}, \cite{Ziem}, \cite{Tai}, where Bessel and Riesz
capacities are used instead of Hausdorff content. In fact, the
exceptional set can be precisely characterized in terms of the
Bessel and Riesz capacities when $f\in\WW^{l,p}(\R^n)$ and
$p>1$.

We extend our results also to the space $\BV_n(\R^n)$ consisting
of functions $v\in \LL^{1}(\R^n)$ such that all its partial
(distributional) derivatives of the $n$-th order are $\R$-valued
Radon measures on $\R^n$ (see Section~\ref{bvs}).

For the plane case $n=2$ these results were obtained in~\cite{BKK}
by
different methods that do not easily extend to the
multidimensional case
$n > 2$ that is the main focus here.

Our proofs rely on the results of \cite{M} on advanced
versions of Sobolev imbedding theorems (see Theorem~\ref{lb2}), of
\cite{Ad} on Choquet integrals of Hardy-Littlewood maximal
functions  with respect to Hausdorff content (see
Theorem~\ref{lb7}), and of \cite{Yom} on the entropy estimate of
near--critical values of differentiable functions (see
Theorem~\ref{lb8}). The key step in the proof of the assertion of
the Morse--Sard Theorem is contained in Lemma~\ref{lb11}.

\section{Preliminaries}

\noindent By  an {\it $n$-dimensional interval} we mean a closed
cube $I=[a,b]^{n}\subset\R^n$ with sides parallel to the
coordinate axis. Furthermore we write $\ell (I) = b-a$ for its
sidelength.

We denote by $\Le^n(F)$ the outer Lebesgue measure of a set
$F\subset\R^n$. Denote by $\H^k$, $\H^{k}_{\infty}$ the
$k$--dimensional Hausdorff measure, Hausdorff content,
respectively: for any $F \subset \R^n$,
$\H^k(F)=\lim\limits_{\alpha\searrow 0}\H^k_\alpha(F)=
\sup_{\alpha >0} \H^{k}_{\alpha}(F)$, where for each $0< \alpha
\leq \infty$,
$$
\H^k_\alpha(F)=\inf\bigl\{ \sum_{i=1}^\infty(\diam F_i)^k\ :\
\diam F_i\le\alpha,\ \ F \subset\bigcup\limits_{i=1}^\infty
F_i\bigr\}.
$$
It is well known that $\H^n(F)\sim\H^n_\infty(F)\sim\Le^n(F)$ for
sets~$F\subset\R^n$.

To simplify the notation, we write $\|f\|_{\LL^1}$ instead of
$\|f\|_{\LL^1(\R^n)}$, etc.

The space $\WW^{k,1} (\R^n)$ is as usual defined as consisting of
those functions $f\in \LL^1(\Omega)$ whose distributional partial
derivatives of order $l\le k$ belongs to $\LL^1(\R^n)$ (for
detailed definitions and differentiability properties of such
functions see, e.g., \cite{EG}, \cite{Ziem}, \cite{Dor}). Denote
by $\nabla^k f$ the vector-valued function consisting of all
$k$-th order partial derivatives of $f$ arranged in some fixed
order.
We use the norm
$$
\|f\|_{\WW^{k,1}}=\|f\|_{\LL^1}+\|\nabla
f\|_{\LL^1}+\dots+\|\nabla^kf\|_{\LL^1}.
$$

Working with Sobolev functions we always assume that the precise
representatives are chosen. If $w\in L^1_{\loc}(\Omega)$, then the
precise representative $w^*$ is defined by
\begin{equation}
\label{lrule}w^*(x)=\left\{\begin{array}{rcl} {\displaystyle
\lim\limits_{r\to
0} \dashint_{B(x,r)}{w}(z)\,\dd z}, & \mbox{ if the limit exists
and is finite,}\\
 0 \qquad\qquad\quad & \; \mbox{ otherwise },
\end{array}\right.
\end{equation}
where the dashed integral as usual denotes the integral mean,
$$
\dashint_{B(x,r)}{ w}(z)dz=\frac{1}{\Le^n(B(x,r))}\int_{B(x,r)}{
w}(z)\,\dd z,
$$
and $B(x,r)=\{y: |y-x|<r\}$ is the open ball of radius $r$
centered at $x$.

The following well-known assertion follows immediately from the
definition of Sobolev spaces.

\begin{lem}\label{lemc3}{\sl
Let $f\in\WW^{l,1}(\R^n)$. Then for any $\varepsilon >0$ there
exist functions $f_0\in C_0^\infty(\R^n)$,
$f_1\in\WW^{l,1}(\R^n)$, such that $f = f_0 +f_1$ \ and \
$\|f_1\|_{\WW^{l,1}}<\varepsilon$.}
\end{lem}

We need a version of the Sobolev Embedding Theorem that gives
inclusions in Lebesgue spaces with respect to suitably general
positive
measures. Very general and precise statements are known, but here
we
restrict attention to the following class of measures:

\begin{df}
\label{db1} Let $\mu$ be a positive measure on $\mathbb R^n$. We
say that $\mu$ has property $(*-l)$ for some $l\le n$,  if
\begin{equation}
\label{cb2} \mu(I) \leq (\ell(I))^{(n-l)}
\end{equation}
for any $n$-dimensional interval $I\subset\mathbb R^n$.
\end{df}

\begin{ttt}[see \cite{M}, \S1.4.3]
\label{lb2} {\sl If $f\in\WW^{l,1}(\R^n)$ and $\mu$ has property
$(*-l)$, then
\begin{equation}
\label{cb3} \int |f| d\mu \leq C\Vert \nabla^lf\Vert_{\LL^1},
\end{equation}
where $C$ does not depend on $\mu,\ f$. }
\end{ttt}

For a function $u\in \LL^1(I)$, $I\subset\R^n$,  define the
polynomial $P_{I,k}[u]$ of degree at most~$k$ by the following
rule:
\begin{equation}
\label{0}\int_Iy^\alpha \left( u(y)-P_{I,k}[u](y) \right) \,\dd
y=0
\end{equation}
for any multi-index $\alpha=(\alpha_1,\dots,\alpha_n)$ of length
$|\alpha|=\alpha_1+\dots+\alpha_n\le k$.

We will often use the following simple technical assertion.

\begin{lem}\label{lb3}{\sl
Suppose $v\in\WW^{n,1}(\R^n)$. Then $v$ is a continuous function
and for any $k=0,\dots,n-1$ and for any $n$-dimensional interval
$I\subset \R^n$ the estimate
\begin{equation}
\label{1} \sup\limits_{y\in I}|v(y)-P_{I,k}[v](y)|\le
C\biggl(\frac{\|\nabla^{k+1}v\|_{\LL^1(I)}}{\ell(I)^{n-k-1}}+
\|\nabla^nv\|_{\LL^1(I)}\biggr)
\end{equation}
holds, where $C$ depends on $n$ only. Moreover, the function
$v_{I,k}(y)=v(y)-P_{I,k}[v](y)$, $y\in I$, can be extended
from~$I$ to the whole of $\R^n$ such that
$v_{I,k}\in\WW^{n,1}(\R^n)$ and
\begin{equation}
\label{1'} \|\nabla^nv_{I,k}\|_{L^1(\R^n)}\le C_0 R(I,k),
\end{equation}
 where $C_0$ also depends on $n$ only and $R(I,k)$ denotes the
right
hand side of the estimates~(\ref{1}) (in brackets). }\end{lem}

\begin{proof}
The existence of a continuous representative for $v$ follows from
Remark~2 of \S1.4.5 in~\cite{M}. Because of coordinate invariance
it is sufficient to prove the estimate~(\ref{1})--(\ref{1'}) for
the case when $I$ is a unit cube: $I=[0,1]^n$. By results of
\cite[\S 1.1.15]{M} for any $u\in \WW^{n,1}(I)$ the estimates
\begin{equation}
\label{1.1} \sup\limits_{y\in I}|u(y)|\le c\|u\|_{\WW^{n,1}(I)}\le
c\bigl( \| P_{I,k}[u]
\|_{\LL^{1}(I)}+\|\nabla^{k+1}u\|_{\LL^1(I)}+\|\nabla
^nu\|_{\LL^1(I)}\bigr),
\end{equation}
hold, where $c=c(n,k)$ is a constant. Taking
$u(y)=v(y)-P_{I,k}[v](y)$, the first term on the right hand side
of (\ref{1.1}) vanishes and so the inequality~(\ref{1.1}) turns to
the estimates~(\ref{1})--(\ref{1'}) (here we used also the
following fact: every function $u\in \WW^{n,1}(I)$ can be extended
to a function $u\in \WW^{n,1}(\R^n)$ such that the estimate $
\|\nabla^nu\|_{\LL^1(\R^n)}\le  c \|u\|_{\WW^{n,1}(I)}$ holds, see
\cite[\S 1.1.15]{M}).
\end{proof}

The following two results are crucial for our proof.

\begin{ttt}[\cite{Ad}]\label{lb7}{\sl
If $f\in\WW^{k,1}(\R^n)$, where $k \in \{ 1,\dots,n-1 \}$, then
$$
\int_0^\infty\H^{n-k}_\infty( \{x\in\R^n : \M
f(x)\ge\lambda\})\,\dd\lambda\le
C\int\limits_{\R^n}{|\nabla^kf(y)|}\,\dd y,
$$
where $C$ depends on $n,k$ only and
$$
\M f(x)=\sup\limits_{r>0}\,r^{-n}\int\limits_{B(x,r)}|f(y)|\,\dd y
$$
is the usual Hardy-Littlewood maximal function of~$f$. }
\end{ttt}

\begin{ttt}[\cite{Yom}]\label{lb8}{\sl
For $A\subset{\R}^m$ and $\varepsilon>0$ let \ $\E(\varepsilon,A)$
denote the minimal number of balls of radius $\varepsilon$
covering $A$. Then for any polynomial $P\colon \R^n\to\R$ of
degree at most $k$, for each ball $B\subset \R^n$ of radius $r>0$,
and any number $\varepsilon >0$ the estimate
$$
\E(\varepsilon r,\{P(x):x\in B,\ |\nabla P(x)|\le\varepsilon\})\le
C_*
$$
holds, where $C_*$ depends on $n,k$ only.}
\end{ttt}

To apply Theorem~\ref{lb7}, we need also the following simple
estimate and its corollary.

\begin{lem}[see Lemma~2 in \cite{Dor}]\label{lb9'}{\sl
Let $u\in\WW^{1,1}(\R^n)$.  Then for any ball $B(z,r)\subset\R^n$,
$B(z,r)\ni x$, the estimate
$$
\biggl|u(x)-\dashint_{B(z,r)}u(y)\,dy\biggr|\le Cr(\M \nabla u)(x)
$$
holds, where $C$ depends on $n$ only and  $\M\nabla u$ is a
Hardy-Littlewood maximal function of~$\nabla u$. }\end{lem}

\begin{cor}\label{lb10}
Let $u\in\WW^{1,1}(\R^n)$. Then for any ball $B\subset \R^n$ of a
radius $r>0$ and for any number $\varepsilon >0$ the estimate
$$
\diam (\{u(x):x\in B,\ (\M \nabla u)(x)\le\varepsilon\})\le
C_{**}\varepsilon r
$$
holds, where $C_{**}$ is a constant depending on $n$ only.
\end{cor}

We will use the following $k$-order analog of Lemma~\ref{lb9'}.

\begin{lem}[see Lemma~2 in \cite{Dor}]\label{dorl2}{\sl
Let $u\in\WW^{k,1}(\R^n)$, $k\le n$.  Then for any $n$-dimensional
interval $I\subset\R^n$, $x\in I$, and for any $m=0,1,\dots,k-1$
the estimate
\begin{equation}
\label{kDor}\bigl|\nabla^mu(x)-\nabla^mP_{I,k-1}[u](x)\bigr|\le
C\ell(I)^{k-m}(\M \nabla^k u)(x)
\end{equation}
holds, where the constant $C$ depends on $n,k$ only. }\end{lem}

\section{On images of sets of small Hausdorff contents}

The main result of this section is the following Luzin
$N$--property for $\WW^{n,1}$--functions:

\begin{ttt}\label{th3.3}{\sl
Let $v\in \WW^{n,1}(\R^n)$. Then for each $\varepsilon>0$ there
exists $\delta>0$ such that for any set $E\subset\R^n$ \ if \
$\H^1_\infty(E)<\delta$, then $\H^1(v(E))<\varepsilon$. In
particular, $\H^1(v(E))=0$ whenever $\H^1(E)=0$.}
\end{ttt}

\noindent For the plane case, $n=2$, Theorem~\ref{th3.3} was
obtained in the paper \cite{BKK}.

\noindent For the remainder of this section we fix a function
$v\in \WW^{n,1}(\R^n)$. To prove Theorem \ref{th3.3},  we need
some preliminary lemmas that we turn to next.

\noindent By {\it a dyadic interval} we understand an interval of
the form
$[\frac{k_1}{2^m},\frac{k_1+1}{2^m}]\times\dots\times[\frac{k_n}{2
^m},\frac{k_n+1}{2^m}]$,
where $k_i,m$ are integers. The following assertion is
straightforward, and hence we omit its proof here.

\begin{lem}\label{lemD}{\sl
For any $n$-dimensional interval $I\subset\R^n$ there exist dyadic
intervals $Q_1,\dots,Q_{2^n}$ such that $I\subset Q_1\cup\dots\cup
Q_{2^n}$ and $\ell(Q_1)=\dots=\ell(Q_{2^n})\le2\ell(I)$.
}\end{lem}

Let  $\{ I_\alpha \}_{\alpha \in A}$ be a family of
$n$-dimensional dyadic intervals. We say that the family $\{
I_\alpha \}$ is $k$-{\it regular}, if for any $n$-dimensional
dyadic interval $Q$ the estimate
\begin{equation}\label{q8}
\ell(Q)^k\ge\sum\limits_{\alpha : I_\alpha\subset
Q}\ell(I_\alpha)^k
\end{equation}
holds.

The next two assertions  are the multidimensional analogs of the
corresponding plane results from the paper \cite{BKK}.

\begin{lem}\label{lb5.1}{\sl
Let $k \in \{ 1,\dots,n \}$ and let  $I_\alpha$ be a family of
$n$-dimensional dyadic intervals. Then there exists
 a $k$-regular family $J_\beta$ of $n$-dimensional
dyadic intervals such that $\cup_\alpha I_\alpha\subset\cup_\beta
J_\beta$ and
$$
\sum\limits_{\beta}\ell(J_\beta)^k\le\sum\limits_{\alpha}\ell(I_\alpha)^k.
$$
}\end{lem}

\begin{proof}
Define
$$
\mathcal F= \left\{ J \, : \, J\subset\mathbb R^n \text { dyadic
interval}; \sum_{I_\alpha\subset J} \ell(I_\alpha)^k\geq \ell(J)^k
\right\}.
$$
Thus $I_\alpha\in \mathcal F$ for each $\alpha$. Denote by
$\mathcal F^* =\{J_\beta\}$ the collection of maximal elements of
$\mathcal F$. Clearly
\begin{equation}
\label{cb9} \bigcup_\alpha I_\alpha \subset \bigcup_\beta J_\beta
,
\end{equation}
and since dyadic intervals are either disjoint or contained in one
another, the $\{J_\beta\}$ are mutually disjoint\footnote{ By {\it
disjoint dyadic intervals} we mean intervals with disjoint
interior.}. It follows that
\begin{equation}
\label{cb10} \sum_\beta\ell(J_\beta)^k\leq
\sum_\beta\sum_{I_\alpha\subset
 J_\beta} \ell(I_\alpha)^k \leq \sum_\alpha\ell
(I_\alpha)^k.
\end{equation}
Observe also that for any dyadic interval $Q\subset\mathbb R^n$,
\begin{equation}
\label{cb11} \sum_{J_\beta\subset Q}\ell(J_\beta)^k \leq
\ell(Q)^k.
\end{equation}
Indeed, if $J_\beta\subset Q$ for some $\beta$, then clearly
either $J_\beta=Q$ or $J_\beta \neq Q$. In the first case the
estimate
is evident, and in the second case we deduce from maximality of
$J_\beta$
that $Q\not\in \mathcal F$, and hence that
\begin{equation*}
 \sum_{J_\beta\subset Q}\ell(J_\beta)^k \leq
\sum_{I_\alpha\subset Q}\ell(I_\alpha)^k < \ell(Q)^k.
\end{equation*}

\end{proof}

\begin{lem}\label{Thh3.3}{\sl
Let $k=0,\dots,n-1$. Then for each $\varepsilon>0$ there exists
$\delta=\delta(\varepsilon,v,k)>0$ such that for any
$(k+1)$-regular family $I_\alpha\subset\R^n$ of $n$-dimensional
dyadic intervals we have if
$\sum_\alpha\ell(I_\alpha)^{k+1}<\delta$, then $\sum_\alpha
R(I_\alpha,k)<\varepsilon$. }\end{lem}

\begin{proof}

Fix $\varepsilon >0$ and let $I_\alpha\subset\R^n$ be a
$(k+1)$-regular family of $n$-dimensional dyadic intervals with
$\sum_\alpha\ell(I_\alpha)^{k+1}<\delta$, where $\delta >0$ will
be specified below. By virtue of Lemma~\ref{lemc3} we can find a
decomposition $v=v_0+v_1$, where $\| \nabla ^jv_{0}
\|_{\LL^{\infty}} \leq K = K(\varepsilon ,v)$ for all
$j=0,1,\dots,n$ \ and \
\begin{equation}\label{cb3''}\| \nabla^nv_{1} \|_{\LL^1} <
\varepsilon.
\end{equation}
 Assume that
\begin{equation}\label{cb8}
\sum_\alpha\ell(I_\alpha)^{k+1}< \delta<
\tfrac{1}{K+1}\varepsilon.
\end{equation}

Define the measure $\mu$ by
\begin{equation}
\label{cb12} \mu = \left(\sum_\alpha \frac
1{\ell(I_\alpha)^{n-k-1}} 1_{I_\alpha}\right) {\mathcal L}^{n},
\end{equation}
where $1_{I_\alpha}$ denotes the indicator function of the
set~$I_\alpha$.

\noindent {\bf Claim.}  $\frac1{2^{n+k+2}}\mu$ has property
$(*-(k+1))$.

\noindent Indeed, write for a dyadic interval $Q$
$$
\mu(Q) =\sum_{I_\alpha\subset Q}
\ell(I_\alpha)^{k+1}+\sum_{Q\subset I_\alpha}\ \frac
{\ell(Q)^n}{\ell(I_\alpha)^{n-k-1}} \leq 2\ell(Q)^{k+1},
$$
where we invoked (\ref{q8}) and the fact that $Q\subset I_\alpha$
for at most one $\alpha$. Then for any interval $I$ we have the
estimate $\mu(I)\le 2^{n+k+2}\ell(I)^{k+1}$ (see
Lemma~\ref{lemD}). This proves the claim.


Now return to the estimate of $\sum_\alpha R(I_\alpha,k)$. In
addition to (\ref{cb8}) we now decrease $\delta >0$ further such
that
\begin{equation}
\label{cb13}
\sum\limits_{\alpha}\|\nabla^nv\|_{\LL^1(I_\alpha)}<\varepsilon/2.
\end{equation}
By definition of $R(I,k)$ (see Lemma~\ref{lb3})  and
properties~(\ref{cb3''}), (\ref{cb3}) (applied to
$f=\nabla^{k+1}v_1,$ $l=n-k-1$\,),
 we have
\begin{eqnarray*}
\sum_{\alpha} R(I_\alpha,k) &=& \sum_{\alpha} \|\nabla^n
v\|_{\LL^1(I_\alpha)}+
\sum_{\alpha} \frac1{\ell (I_\alpha)^{n-k-1}}\int_{I_\alpha}\!
|\nabla^{k+1} v|\\
&\leq&  \varepsilon/2  +\tfrac{K}{K+1}\varepsilon
+\sum_{\alpha} \frac1{\ell (I_\alpha)^{n-k-1}}\int_{I_\alpha} \!
|\nabla^{k+1}v_1|\\
&=& C'\varepsilon +C\int \! |\nabla^{k+1}v_1| \,d\mu \leq
C''\varepsilon .
\end{eqnarray*}
Since $\ve >0$ was arbitrary, the proof of Lemma~\ref{Thh3.3} is
complete.
\end{proof}

\begin{proof}[Proof of Theorem~\ref{th3.3}] We have an
obvious estimate $\diam v(I)\le C R(I,0)$ for any $n$-dimensional
interval~$I\subset\R^n$ (see Lemma~1.4). Fix $\varepsilon>0$ and
take $\delta=\delta(\varepsilon)$ from Theorem~\ref{Thh3.3} for
$k=0$, i.e., for any $1$-regular family $I_\alpha\subset\R^n$ of
$n$-dimensional dyadic intervals we have if
$\sum_\alpha\ell(I_\alpha)<\delta$, then $\sum_\alpha
R(I_\alpha,0)<\varepsilon$, consequently, $\sum_\alpha \diam
v(I_\alpha)<C\varepsilon$.  Now the assertion of
Theorem~\ref{th3.3} follows from Lemmas~2.2--2.3 (by these Lemmas,
there exists $\delta_1>0$ such that if $\H^1_\infty(E)<\delta_1$,
then $E$ can be covered by a $1$-regular family
$I_\alpha\subset\R^n$ of $n$-dimensional dyadic intervals with
$\sum_\alpha\ell(I_\alpha)<\delta$).
\end{proof}

\section{On approximation of $\WW^{l,1}$ Sobolev functions}

\begin{ttt}\label{Th_ap}{\sl
Let $k,l\in\{1,\dots,n\}$, $k\le l$. Then for any
$f\in\WW^{l,1}(\R^n)$ and for each $\varepsilon>0$ there exist an
open set $U\subset\R^n$ and a function $g\in C^k(\R^n)$ such that
$\H^{n-l+k}_\infty(U)<\varepsilon$ and $f\equiv g$, $\nabla^mf
\equiv \nabla^mg$ on $\R^n\setminus U$ for $m=1,\dots,k$.}
\end{ttt}

The proof of Theorem~\ref{Th_ap} is based on the results of
\cite{Ad}, \cite{Dor} and on the classical Whitney Extension
Theorem:

\begin{ttt}
\label{TapW} {\sl Let $k\in\mathbb N$ and let $f=f_0, f_\alpha$ be
a finite family of functions defined on the closed set
$E\subset\R^n$, where $\alpha$ ranges over all multi-indices
$\alpha=(\alpha_1,\dots,\alpha_n)$ with $|\alpha|=\alpha_1+\dots
+\alpha_n\le k$. For $x,y\in E$ and a multi-index $\alpha$,
$|\alpha|\le k$, put
$$
T_\alpha(x;y)=\sum\limits_{|\beta|\le
k-|\alpha|}\frac1{\beta!}f_{\alpha+\beta}(x)\cdot(y-x)^\beta,
$$
$$
R_\alpha(x;y)=f_\alpha(y)-T_\alpha(x;y).
$$
Suppose that there exists a function $\omega\colon [0,+\infty)\to
[0,+\infty)$ such that $\omega(t)\to 0$ as $t\searrow 0$ and for
each multi-index $\alpha$, $|\alpha|\le k$, and for all $x,y\in E$
the estimate
\begin{equation}
\label{ape1} |R_\alpha(x;y)|\le \omega(|x-y|)|x-y|^{k-|\alpha|}
\end{equation}
holds. Then there exists  a function $g\in C^k(\R^n)$ such that
$f\equiv g$, $f_\alpha \equiv \partial^\alpha g$ \ on $E$ \ for
$|\alpha|\le k$.}
\end{ttt}

\begin{proof}[Proof of Theorem~\ref{Th_ap}]
Let the assumptions of Theorem~\ref{Th_ap} be fulfilled. For the
case $k=l$ the assertion of the Theorem is well-known (see, e.g.,
\cite{Liu}, \cite{Boj}, or \cite{Ziem}).

Now fix $k<l$. Then the gradients $\nabla^m f (x)$, $m\le k$, are
well-defined for all $x\in\R^n\setminus A_{k}$, where
$\H^{n-l+k}(A_{k})=0$ (see \cite{Dor}). For a multi-index $\alpha$
with $|\alpha|\le k$ denote by $T_\alpha(f,x;y)$ the Taylor
polynomial of order at most $k-|\alpha|$ for the
function~$\partial^\alpha f$ with the center at~$x$:
$$
T_\alpha(f,x;y)=\sum\limits_{|\beta|\le
k-|\alpha|}\frac1{\beta!}\partial^{\alpha+\beta}f(x)\cdot(y-
x)^\beta.
$$
By virtue of the Whitney Extension Theorem~\ref{TapW}, we finish
the proof of Theorem~\ref{Th_ap} by checking that for each
multi-index $\alpha$ with $|\alpha|\le k$ the corresponding Taylor
remainder term satisfies the estimate $\partial^\alpha
f(y)-T_\alpha(f,x;y)=o(|x-y|^{k-|\alpha|})$ uniformly for
$x,y\in\R^n\setminus U$, where $\H^{n-l+k}_\infty(U)$ is small.

Take a sequence $f_i\in C^\infty_0(\R^n)$ such that
$$
\|\nabla^lf_i-\nabla^lf\|_{\LL^1}< 4^{-i}.
$$
Denote $\tilde f_i=f-f_i$. Put
$$
B_i=\{x\in\R^n: (\M \nabla^k\tilde f_i)(x)>2^{i}\},
$$
$$
G_i=A_{k}\cup\left(\bigcup\limits_{j=i}^\infty B_j\right).
$$
Then by Theorem~\ref{lb7} we have
$$
\H^{n-l+k}_\infty(B_i)\le c\,2^{-i},
$$
and consequently,
$$
\H^{n-l+k}_\infty(G_i)< C\,2^{-i}.
$$
By construction,
\begin{equation}\label{l2.5}
|\nabla^k \tilde f_j(x)|\le2^{-j}
\end{equation}
for all $x \in \R^n \setminus G_i$ and all $j \geq i$. For a
multi-index $\alpha$ with $|\alpha|\le k-1$ denote by
$T_{\alpha,k-1}(f,x;y)$ the Taylor polynomial of order
$k-1-|\alpha|$ for the function~$\partial^\alpha f$ with the
center at~$x$:
$$
T_{\alpha,k-1}(f,x;y)=\sum\limits_{|\beta|\le
k-1-|\alpha|}\frac1{\beta!}\partial^{\alpha+\beta}f(x)\cdot(y-
x)^\beta.
$$
In our notation,
$$
T_{\alpha}(f,x;y)=T_{\alpha,k-1}(f,x;y)+\sum\limits_{|\beta|=k-
|\alpha|
}\frac1{\beta!}\partial^{\alpha+\beta}f(x)\cdot(y-x)^\beta.
$$
We start by estimating the remainder term $\partial^\alpha\tilde
f_j(y)-T_{\alpha,k-1}(\tilde f_j,x;y)$ for a multi-index~$\alpha$
with $|\alpha|\le k-1$. Fix $x,y\in\R^n\setminus G_i$, \ $j\ge i$,
and an $n$-dimensional interval~$I$ such that $x,y\in I$,
$|x-y|\sim \ell(I)$.  By construction and Lemma~\ref{dorl2},
$$
\bigl|\partial^\alpha\tilde f_j(y)-
\partial^\alpha P_{I,k-1}[\tilde f_j](y)\bigr|\le C\ell(I)^{k-
|\alpha|}(\M
\nabla^k \tilde f_j)(y)\le C|x-y|^{k-|\alpha|}2^{-j}.
$$
For the same reasons we find for any multi-index $\beta$ with
$|\beta|\le k-1-|\alpha|$ that
$$
\bigl|\partial^{\alpha+\beta}\tilde f_j(x)-
\partial^{\alpha+\beta} P_{I,k-1}[\tilde f_j](x)\bigr|\le
C\ell(I)^{k-|\alpha|-|\beta|}(\M
\nabla^k \tilde f_j)(x)\le C|x-y|^{k-|\alpha|-|\beta|}2^{-j}.
$$
Consequently,
\begin{eqnarray}
|\partial^\alpha\tilde{f}_j(y)-T_{\alpha,k-1}(\tilde{f}_j,x;y)|
&\leq& \bigl|\partial^\alpha\tilde f_j(y)-\partial^\alpha
P_{I,k-1}[\tilde f_j](y)\bigr| \\ \nonumber &&
+\bigl|\partial^\alpha P_{I,k-1}[\tilde
f_j](y)-T_{\alpha,k-1}(\tilde f_j,x;y)\bigr|\\ \nonumber &\leq&
C|x-y|^{k-|\alpha|}2^{-j}\\ \nonumber && + \sum\limits_{|\beta|\le
k-1-|\alpha|}\frac1{\beta!}\bigl|\bigl(\partial^{\alpha+\beta}
\tilde{f}_j(x)-\partial^{\alpha+\beta}
P_{I,k-1}[\tilde{f}_j](x)\bigr)\cdot(y-x)^\beta\bigr|\\ \nonumber
&\leq& C_1|x-y|^{k-|\alpha|}2^{-j}.\label{ap3}
\end{eqnarray}
Finally from the last estimate and (\ref{l2.5}) we have
\begin{eqnarray*}
|\partial^\alpha f(y)-T_\alpha(f,x;y)| &\leq& | \partial^\alpha
\tilde f_j(y)-T_\alpha(\tilde f_j,x;y)| + |
\partial^\alpha f_j(y)-T_\alpha( f_j,x;y)|\\
&\leq& | \partial^\alpha \tilde f_j(y)-T_{\alpha,k-1}(\tilde
f_j,x;y)|+ |\nabla^k\tilde f_j(x)|\cdot|x-y|^{k-|\alpha|}\\
&& +\omega_{f_j}(|x-y|)\cdot|x-y|^{k-|\alpha|}\\
&\le& C'|x-y|^{k-|\alpha|}2^{-j}+\omega_{f_j}(|x-y|)\cdot|x-y|^{k-
|\alpha|}\\
&=&
\bigl(C'2^{-j}+\omega_{f_j}(|x-y|)\bigr)\cdot|x-y|^{k-|\alpha|},
\end{eqnarray*}
where $\omega_{f_j}(r)\to 0$ as $r\to0$ (the latter holds since
$f_j\in C^\infty_0(\R^n)$\,). We emphasize that the last
inequality is valid for all $j\ge i$ and $x,y\in \R^n\setminus
G_i$. Take an open set $U_i\supset G_i$ such that
\begin{equation*}
\H^{n-l+k}_\infty(U_i)< C\,2^{-i}.
\end{equation*}
Put $E_i=\R^n\setminus U_i$. Then by construction
\begin{equation}\label{ap5}
|\partial^\alpha f(y)-T_\alpha(f,x;y)|\le
\bigl(C'2^{-j}+\omega_{f_j}(|x-y|)\bigr)\cdot|x-y|^{k-|\alpha|}
\end{equation}
for all $j\ge i$,  $|\alpha|\le k$,  and $x,y\in E_i$. Then the
assumptions of the Whitney Extension Theorem~\ref{TapW} are
fulfilled, and hence the proof of Theorem~\ref{Th_ap} is complete.

\end{proof}

\begin{rem}\label{remapp}
Using the extension formula and the methods from the proof of
Theorem~\ref{bv''}  (see Section~6 below; this approach was
originally introduced in \cite{Boj}\,), one can prove that for
$k<l$ the function~$g$ from the assertion of Theorem~\ref{Th_ap}
can be constructed such that in addition the estimate
$\|f-g\|_{\WW^{k+1,1}}<\varepsilon$ holds.
\end{rem}

\section{Morse--Sard theorem in $\WW^{n,1}(\R^n)$}

Recall that if $v\in\WW^{n,1}(\R^n)$ and $k=1,\dots,n$, then
$\nabla^k v(x)$ is well-defined for $\mathcal H^{k}$-almost all
$x\in\R^n$ (see \cite{Dor}). In particular, $v$ is differentiable
(in the classical Fr\'{e}chet sense) and the classical derivative
coincides with $\nabla v(x)=\lim\limits_{r\to 0}
\dashint_{B(x,r)}{\nabla v}(z)\,\dd z$ at all points
$x\in\R^n\setminus A_{v}$, where $\H^1(A_v)=0$. Consequently, in
view of Theorem~\ref{th3.3}, $\H^1(v(A_v))=0$.

Denote $Z_v=\{x\in \R^n\setminus A_v:\nabla v(x)=0\}$. The main
result of the section is as follows:

\begin{ttt}
\label{MS} {\sl  If $v\in \WW^{n,1}(\R^n)$, then
$\H^1(v(Z_v))=0$.}
\end{ttt}

\noindent For the remainder of the section we fix a function $v\in
\WW^{n,1}(\R^n)$.

\noindent The key point of the proof is contained in the following
lemma.

\begin{lem}\label{lb11}{\sl
For any $n$-dimensional dyadic interval $I\subset\R^n$ the
estimate
\begin{equation}
\label{6} \H^1(v(Z_v\cap I))\le C \|\nabla^nv\|_{L^1(I)}
\end{equation}
holds, where $C$ depends on $n$ only.}\end{lem}

\begin{proof}
Note that by formula (\ref{1'}) it is sufficient to prove the
estimate
\begin{equation}
\label{6corr} \H^1(v(Z_v\cap I))\le C
\|\nabla^nv_{I,n-1}\|_{L^1(\R^n)},
\end{equation}
where the function $v_{I,n-1}$ was defined in Lemma~\ref{lb3}.

Fix an $n$-dimensional dyadic interval $I\subset\R^n$. To simplify
the notation,  we will write $v_I$ and $P_I$ instead of
$v_{I,n-1}$ and $P_{I,n-1}[v]$ respectively. In particular,
$v_{I}(x)=v(x)-P_{I}(x)$ for all $x\in I$. Denote
$$
\sigma=\|\nabla^nv_{I}\|_{L^1(\R^n)},\quad E_j=\{x\in\R^n: (\M
\nabla v_{I})(x)\in(2^{j-1},2^{j}]\},\quad j\in \mathbb Z.
$$
Denote also $\delta_j=\H^1_\infty(E_j).$ Then by
Theorem~\ref{lb7},
$$
\sum\limits_{j=-\infty}^\infty \delta_j2^j\le C_1\sigma,
$$
where $C_1$ depends on $n$ only. By construction, for each
$j\in\mathbb Z$ there exists a family of balls $B_{ij}\subset\R^n$
of radii $r_{ij}$ such that
$$
E_j\subset \bigcup\limits_{i=1}^\infty B_{ij} \mbox{\ \ \ and\ \ \
}\sum\limits_{i=1}^\infty r_{ij}\le 3\delta_j.
$$
Denote
$$
Z_{ij}=Z_v\cap I\cap E_j\cap B_{ij} \quad \mbox{ and } \quad
Z_\infty= Z_v\cap
I\setminus\biggl(\bigcup\limits_{i,j}Z_{ij}\biggr).
$$
By construction $Z_\infty\subset\{x\in\R^n:(\M \nabla
v_{I})(x)=\infty\}$, so by Theorem~\ref{lb7}, \ $\H^1(Z_\infty)=0$
and hence by Theorem~\ref{th3.3}, $\H^1(v(Z_\infty))=0$. Thus it
is sufficient to estimate $\H^1(v(Z_{ij}))$.

Since $\nabla P_{I}(x)=-\nabla v_{I}(x)$ at each point $x\in
Z_v\cap I$, we have by construction for all $i$, $j$:
$$
Z_{ij}\subset \{x\in B_{ij}: |\nabla P_{I}(x)|=|\nabla
v_{I}(x)|\le(\M \nabla v_{I})(x)\le2^j\}.
$$
Applying Theorem~\ref{lb8} and Corollary~\ref{lb10} to functions
$P_{I}$, \ $v_{I}$, respectively, with $B=B_{ij}$ and $\varepsilon
=2^j$, we find a finite family of balls  $T_k\subset\R$ each of
radius $(1+C_{**})2^jr_{ij}$, $k=1,\dots,C_*$, such that
$$
\bigcup\limits_{k=1}^{C_*}T_k\supset v(Z_{ij}).
$$
Therefore
$$
\H^1(v(Z_{ij}))\le 2C_*(1+C_{**})2^{j}r_{ij},
$$
and consequently,
$$
\H^1(v(Z_{v}\cap I))\le
\sum\limits_{j=-\infty}^{\infty}\sum\limits_{i=1}^\infty
2C_*(1+C_{**})2^jr_{ij}\le
6C_*(1+C_{**})\sum\limits_{j=-\infty}^{\infty}2^j\delta_j\le
C'\sigma.
$$
The last estimate finishes the proof of the Lemma.
\end{proof}

From the last result and the absolute continuity of the Lebesgue
integral we infer

\begin{cor}
\label{cor00} {\sl  For any $\varepsilon>0$ there exists
$\delta>0$ such that for any set $E\subset\R^n$ if
$\H^n_\infty(E)\le\delta$, then $\H^1(v(Z_v\cap
E))\le\varepsilon$. In particular, $\H^1(v(Z_v\cap E))=0$ for any
$E\subset\R^n$ with $\H^n_\infty(E)=0$.}
\end{cor}

Because of the classical Morse-Sard Theorem for $g\in C^n(\R^n)$,
Theorem~\ref{Th_ap} (applied to the case $k=n$ ) implies

\begin{cor}
\label{Th_ap2} {\sl  There exists a set $Z_{0,v}$ of
$n$-dimensional Lebesgue measure zero such that
$\H^1(v(Z_v\setminus Z_{0,v}))=0$. In particular,
$\H^1(v(Z_v))=\H^1(v(Z_{0,v}))$.}
\end{cor}

From Corollaries~\ref{Th_ap2}, \ref{cor00} we conclude the proof
of Theorem~\ref{MS}.

\section{Application to the level sets of $\WW^{n,1}$ functions}

Theorem \ref{Th_ap} for the case $k=1$ implies the following

\begin{ttt}
\label{Th_ap3} {\sl  Let $v\in\WW^{n,1}(\R^n)$. Then for any
$\varepsilon>0$ there exist an open set $U\subset\R^n$ and a
function $g\in \CC^1(\R^n)$ such that $\H^1_\infty(U)<\varepsilon$
and $v\equiv g$, $\nabla v \equiv \nabla g$ on $\R^n\setminus U$.
}
\end{ttt}

If we apply Theorems~\ref{th3.3}, \ref{MS} to the last assertion,
we obtain

\begin{cor}
\label{cor2.4} {\sl Let $v\in\WW^{n,1}(\R^n)$. Then for any
$\varepsilon>0$ there exist an open set $V\subset\R$ and a
function $g\in \CC^1(\R^n)$ such that $\H^1(V)<\varepsilon$,
$v(A_v)\subset V$ and $v|_{v^{-1}(\R\setminus
V)}=g|_{v^{-1}(\R\setminus V)}$, $\nabla v|_{v^{-1}(\R\setminus
V)}=\nabla g|_{v^{-1}(\R\setminus V)}\ne0$. }
\end{cor}

Finally we have

\begin{ttt}
\label{Th2.1} {\sl Let $v\in\WW^{n,1}(\R^n)$. Then for almost all
$y\in\R$ the preimage $v^{-1}(y)$ is a finite disjoint family of
$(n-1)$-dimensional $\CC^1$-smooth compact manifolds (without
boundary) $S_j$, $j=1,\dots,N(y)$. }
\end{ttt}

\begin{proof}
The inclusion $v\in \WW^{n,1}(\R^n)$ and Lemma~\ref{lb3} easily
imply the following statement:

\begin{itemize}
\item[(i)] For any $\varepsilon>0$ there exists $R_\varepsilon\in
(0,+\infty)$ such that $|v(x)|<\varepsilon$ for all
$x\in\R^n\setminus B(0, R_\varepsilon)$.
\end{itemize}

Fix arbitrary $\varepsilon>0$. Take the corresponding set $V$ and
function~$g\in C^1(\R^n)$ from Corollary~\ref{cor2.4}. Let $0\ne
y\in v(\R^n)\setminus V$. Denote $F_v=v^{-1}(y)$, $F_g=g^{-1}(y)$.
We assert the following properties of these sets.

\begin{itemize}
\item[(ii)]  $F_v$ is a compact set;

\item[(iii)] $F_v\subset F_g$;

\item[(iv)] $\nabla v=\nabla g\ne 0$ on $F_v$;

\item[(v)] The function $v$ is differentiable (in the classical
sense) at each~$x \in F_v$, and the classical derivative coincides
with~$\nabla v(x)=\lim\limits_{r\to 0} \dashint_{B(x,r)}{\nabla
v}(z)\,\dd z$.
\end{itemize}

\noindent Indeed, (ii) follows from~(i), (iii)-(iv) follow from
Corollary~\ref{cor2.4}, and (v) follows from the condition
$v(A_v)\subset V$ of Corollary~\ref{cor2.4}.

\noindent We require one more property of these sets:

\begin{itemize}
\item[(vi)] For any $x_0\in F_v$ there exists $r>0$ such that
$F_v\cap B(x_0,r)=F_g\cap B(x_0,r)$.
\end{itemize}

\noindent Indeed, take any point $x_0\in F_v$ and suppose the
claim~(vi) is false. Then there exists a sequence of points
$F_g\setminus F_v\ni x_i\to x_0$. Denote by $I_x$ the straight
line segment of length $r$ with the center at~$x$ parallel to the
vector $\nabla v(x_0)=\nabla g(x_0)$. Evidently, for sufficiently
small $r>0$ the equality $I_x\cap F_g=\{x\}$ holds for any~$x\in
F_g\cap B(x_0,r)$. Then by construction $I_{x_i}\cap
F_v=\emptyset$ for sufficiently large $i$. Hence for sufficiently
large $i$ either $v>y$ on $I_{x_i}$ or $v<y$ on $I_{x_i}$. For
definiteness, suppose $v>y$ on $I_{x_i}$ for all $i\in\mathbb N$.
In the limit we obtain the inequality $v\ge y=v(x_0)$ on
$I_{x_0}$. But the last assertion contradicts~(iv)--(v). This
contradiction finishes the proof of~(vi).

Obviously, (ii)-(vi) imply that each connected component of the
set $F_v=v^{-1}(y)$ is a compact $(n-1)$-dimensional $C^1$-smooth
manifold (without boundary).
\end{proof}

\section{On the case of $\BV_n$ functions}
\label{bvs}

For signed or vector--valued Radon measures $\mu$ we denote by
$\|\mu\|$ the total variation measure of $\mu$. The space $\BV_k
(\R^n)$ is as usual defined as consisting of those functions $f\in
W^{k-1,1}(\R^n)$ whose distributional partial derivatives of order
$k$ are Radon measures with $\|D^kf\|(\R^n)<\infty$, where we
denote by $D^kf$ the vector-valued measure consisting of all
$k$-order partial derivatives of $f$ (for detailed definitions and
differentiability properties of such functions see, e.g.,
\cite{EG}, \cite{Ziem}, \cite{Dor}). In particular, the following
fact is well-known.

\begin{lem}\label{bv}{\sl
Let $f\in \BV_k(\R^n)$. Then there exists a sequence $f_i\in
\CC^\infty_0(\R^n)$ such that
$$
\|f_i-f\|_{\WW^{k-1,1}}\to 0 \ , \quad  \|\nabla^k
f_i\|_{\LL^1}\le C, \mbox{ and}
$$
$$
\|\nabla^k f_i\|_{\LL^1(U)}\to \|D^kf\|(U)
$$
for any open subset $U \subset \R^n$ with $\| D^{k}f \| (\partial
U)=0$.}
\end{lem}

The results obtained in the previous sections were established for
functions of class $\BV_2(\R^2)$ in~\cite{BKK}, hence in the
present section we consider only functions of class $\BV_n(\R^n)$
for $n\ge3$. Recall that in this case $\nabla^k v(x)$ is
well-defined for $\mathcal H^{k}$-almost all $x\in\R^n$,
$k=1,\dots,n-2$ (see \cite{Dor}). In particular, $v$ is
differentiable (in the classical Fr\'{e}chet sense) at all points
$x\in\R^n\setminus A_{v}$, where $\H^1(A_v)=0$. Denote $Z_v=\{x\in
\R^n\setminus A_v:\nabla v(x)=0\}$. Most of the results from the
previous sections remain valid for functions $v\in \BV_n(\R^n)$.
More precisely, Theorem~\ref{th3.3}, Lemma~\ref{Thh3.3} for $k\le
n-2$, Theorem~\ref{MS}, Theorems~\ref{Th_ap3} and \ref{Th2.1} are
also true in this more general $\BV$ context. Except for
Theorem~\ref{MS}, whose proof we discuss below, the proofs are
entirely analogous. Also, the assertion of approximation
Theorem~\ref{Th_ap} remains valid for $f\in\BV_l(\R^n)$,
$k,l\in\{1,\dots,n\}$, $k\le l$, $k\ne l-1$ (for the case $k=l$ it
follows immediately  from the results of~\cite{Dor} and
\cite{Liu}; the proof for $k\le l-2$ will be discussed below).

On the other hand, the assertion of Lemma~\ref{Thh3.3} for $k=
n-1$ is not valid for a general $v\in \BV_n(\R^n)$. Also the
assertion of the Approximation Theorem~\ref{Th_ap} is not valid
for $f\in\BV_l(\R^n)$ when $k= l-1$.

To prove the assertion of the Approximation Theorem~\ref{Th_ap}
for $f\in\BV_l(\R^n)$, $k,l\in\{1,\dots,n\}$ when $k\le l-2$, one
can repeat the arguments from the proof for the Sobolev case (see
Section~3). Proceeding in this manner, one notices that in the
first step it is necessary to have a sequence of functions $f_i\in
\CC^k(\R^n)$ with $\|f-f_i\|_{\BV_l}\to0$. Such a sequence exists
because of the following result.

\begin{ttt}\label{bv''}{\sl
Let $f\in \BV_l(\R^n)$, $l\le n$. Then for any $\varepsilon>0$
there exists a function $g\in  \BV_l(\R^n)$ such that
\begin{itemize}

\item[(i)] \ $\|f-g\|_{\BV_l}<\varepsilon$;

\item[(ii)] $g\in\CC^{l-2,1}(\R^n),$ i.e., $g\in\CC^{l-2}(\R^n)$
and $\nabla^{l-2}g$ is a Lipschitz function;

\item[(iii)] \ there exist an open set $U\subset\R^n$  such that
$\H^{n-1}_\infty(U)<\varepsilon$  and $f\equiv g$, $\nabla^mf
\equiv \nabla^mg$ on $\R^n\setminus U$ for $m=1,\dots,l-2$.
\end{itemize}
}
\end{ttt}

Very similar results were proved in \cite{Boj} for the case of
Sobolev functions $f\in\WW^{l,p}(\R^n)$ with $p>1$, and our proof
follows the ideas from~\cite{Boj}.

To prove Theorem~\ref{bv''}, we need some preliminary results.

\begin{lem}\label{bv-1}{\sl
Let $f\in \BV_l(\R^n)$, $l\le n$. Then for each $\varepsilon>0$
there exists $\delta>0$ such that for any open set $U\subset\R^n$
we have that if $\H^{n-1}_\infty(U)<\delta$ then
$\|D^lf\|(U)<\varepsilon$. }
\end{lem}

\begin{proof}
It is an easy consequence of the Coarea Formula, and we leave
details to the interested reader (or see \cite[Lemma 2.4]{BKK}).
\end{proof}

\begin{rem}\label{remdor}
Using the methods of the proof of Lemma~2 in \cite{Dor}, one can
prove the following result. Let $u\in\BV_{k+1}(\R^n)$, $k+1\le n$.
Then for any $n$-dimensional interval $Q\subset\R^n$ and any point
$x\in\R^n$ with $\dist(x,Q)\le 9n\ell(Q)$  the estimates
\begin{equation}
\label{kkDor} \bigl|\nabla^kP_{Q,k}[u](x)\bigr|\le C(\M \nabla^k
u)(x),\ \quad\bigl|\nabla^mu(x)-\nabla^mP_{Q,k}[u](x)\bigr|\le
C\ell(I)^{k-m}(\M \nabla^k u)(x)
\end{equation}
hold for each $m \in \{ 0, \, \dots \, , \, k-1 \}$, where the
constant $C$ depends on $n$ only.
\end{rem}

\begin{proof}[Proof of Theorem~\ref{bv''}]
Fix $\varepsilon\in(0,1)$. Let $U$ be an open set such that
\begin{equation}\label{apbv0}\H^{n-1}_\infty(U)<\varepsilon,
\end{equation}
\begin{equation}\label{apbv-1}\|D^lf\|(U)<\varepsilon,
\end{equation}
\begin{equation}\label{apbv1}(\M
\nabla^{l-1}f)(x)\le C_\varepsilon\quad \forall x\in\R^n\setminus
U
\end{equation}
The existence of $U$ follows from~Theorem~\ref{lb7}, that remains
valid for $f \in \BV_{l}( \R^n )$ provided the $\LL^1$ norm is
replaced by the total variation norm (see \cite{Ad}), and
Lemma~\ref{bv-1}. Denote $F=\R^n\setminus U$. Take a Whitney cube
decomposition $U=\bigcup\limits_{j=1}^\infty{Q_j}$, where all
cubes $Q_j$ are dyadic, and select an associated smooth partition
of unity $\{\varphi_j\}_{j\in\mathbb N}$. Recall the standard
properties of $Q_j$, $\varphi_j$ (see \cite[Chapter VI]{St}\,):

\begin{itemize}

\item[(i)] \ $\diam(Q_j)\le \dist(Q_j,F)\le 4\diam(Q_j)<1$;

\item[(ii)] \ every point $x\in U$ is covered by at most
$N=(12)^n$ different cubes $Q^*_j$, where the cube $Q^*_j$ has the
same center as $Q_j$ and $\ell(Q^*_j)=\frac98\ell(Q_j)$;

\item[(iii)] \ for each $j\in\mathbb N$ \ \ $\supp
\varphi_j\subset Q^*_j\subset U$, \,moreover,
$|\nabla^m\varphi_j|\le C_m\ell(Q_j)^{-m}$ for all $m\in \mathbb
N$;

\item[(iv)] \ all $\varphi_j\ge0$ \ and \
$\sum\limits_{j=1}^\infty \varphi_j(x)\equiv 1$ \ on $U$.
\end{itemize}
Now we define the function $g\colon \R^n\to\R$ by

\begin{equation}\label{apbv2} g(x)=\left\{\begin{array}{rcl}
f(x),\quad x\in F; & \\
\sum\limits_{j=1}^\infty \varphi_j(x)P_{Q^*_j,l-1}[f](x), \ \ x\in
U.
\end{array}\right.
\end{equation}
Recall the following properties of the polynomials
$P_{Q^*_j,l-1}[f](x)$ (see~\cite[page 1034]{Dor}):
\begin{equation}\label{apbv3}
\begin{array}{rcl}
\int\limits_{Q^*_j}|\nabla^mf(z)-\nabla^mP_{Q^*_j,l-1}[f](z)|\,\dd
z\le C\ell(Q^*_j)^{l-m}\|D^lf\|(Q^*_j);\\
\|D^l(f-P_{Q^*_j,l-1}[f])\|(Q^*_j)= \|D^lf\|(Q^*_j),
\end{array}
\end{equation}
where $m \in \{ 0, 1, \, \dots \, , \, l-1 \}$. From these
properties and assumption~(\ref{apbv-1}) we get by direct
calculation for each $m=0,\dots, l-1$ the estimates
\begin{equation}\label{apbv4'`}
\sum\limits_{j=1}^\infty
\|\nabla^m(\varphi_j(f-P_{Q^*_j,l-1}[f]))\|_{\LL^1(Q^*_j)}\le
C\|D^lf\|(U)<C\varepsilon.
\end{equation}
Analogously,
\begin{equation}\label{apbv4''`} \sum\limits_{j=1}^\infty
\|D^l(\varphi_j(f-P_{Q^*_j,l-1}[f]))\|(Q^*_j)\le
C\|D^lf\|(U)<C\varepsilon.
\end{equation}

From the convergence of the above series and from the completeness
of the space $\BV_l(\R^n)$ it follows readily that
$f-g=\sum\limits_{j=1}^\infty
\varphi_j(f-P_{Q^*_j,l-1}[f])\in\BV_l(\R^n)$. Consequently,
\begin{equation}\label{apbv4'''}g\in\BV_l(\R^n)
\end{equation}
and
\begin{equation}\label{apbv4'}
\|f-g\|_{\BV_l}<C\varepsilon.
\end{equation}
Thus to finish the proof of the Theorem, it is sufficient to check
that
\begin{equation}\label{apbv10}
\|\nabla^{l-1}g\|_{\LL^\infty}<\infty.
\end{equation}
From~(\ref{apbv1}) by construction it follows that
\begin{equation}\label{apbv11}
\esssup_F|\nabla^{l-1}g|=\esssup_F|\nabla^{l-1}f|\le
C_\varepsilon.
\end{equation}
Now estimate $|\nabla^{l-1}g(y)|$ for $y\in U$. Let $y\in
Q_{j_0}$. Take $x\in F$ such that
$\dist(x,Q_{j_0})=\dist(F,Q_{j_0})$. Then
$C_0\ell(Q^*_j)\le|y-x|\le C_1\ell(Q^*_j)$ for each $Q^*_j\ni y$.
Consider the $(l-2)$-order Taylor polynomial
$$T(f,x;y)=\sum\limits_{|\beta|\le
l-2}\frac1{\beta!}\partial^{\beta}f(x)\cdot(y-x)^\beta.
$$
From assumption~(\ref{apbv1}) and Remark~\ref{remdor} (with
$k=l-1$) it follows that for arbitrary multi-index $\alpha$ with
$|\alpha|\le l-1$
\begin{eqnarray}
\nonumber|\partial^\alpha(P_{Q^*_j,l-1}[f](y)-T(f,x;y))|&\leq&
\bigl|\nabla^{l-1}
P_{Q^*_j,l-1}[f](x)\bigr|\cdot|x-y|^{l-1-\alpha}
\\
\nonumber && + \sum\limits_{|\beta|\le
l-2-|\alpha|}\frac1{\beta!}\bigl|\bigl(
\partial^{\alpha+\beta}
P_{Q^*_j,l-1}[f](x)-\partial^{\alpha+\beta}{f}(x)\bigr)\cdot(y-
x)^\beta\bigr|\\
 &\leq& C_2|x-y|^{l-1-|\alpha|}\le
C_3\ell(Q^*_j)^{l-1-|\alpha|}.\label{apbv13}
\end{eqnarray}
From the last estimate we have
\begin{eqnarray}
\nonumber |\nabla^{l-1}g(y)|&=&|\nabla^{l-1}(g(y)-T(f,x;y))|=
\bigl|\sum\limits_{j:Q^*_j\ni
y}\nabla^{l-1}\bigl(\varphi_j(y)(P_{Q^*_j,l-1}[f](y)-
T(f,x;y))\bigr)
\bigr| \\&& \le \sum\limits_{j:Q^*_j\ni y}\sum\limits_{m=0}^{l-1}|
\nabla^{l-1-m}\varphi_j(y)|
\cdot|\nabla^{m}(P_{Q^*_j,l-1}[f](y)-T(f,x;y))| \le C_4,\ \ \
\label{apbv15}
\end{eqnarray}
where the constant $C_4$ does not depend on $y\in U$. The last
estimate finishes the proof of the target
assertion~(\ref{apbv10}).
\end{proof}

Now we discuss the proof of Theorem~\ref{MS} in the $\BV$-case,
which is more delicate than the Sobolev case.

The assertion of the key Lemma~\ref{lb11} remains valid for $v\in
\BV_n(\R^n)$ with identical proof if we replace in its formulation
$\|\nabla^n v\|_{\LL^1(I)}$ by $\|D^nv\|(I)\sim\|D^n
v_{I,n-1}\|(\R^n)$. From the last fact using the standard covering
lemmas one can easily deduce the following:

\begin{lem}\label{bv1}{\sl  Let $v\in \BV_n(\R^n)$. Then for each
$\varepsilon>0$ there exists $\delta>0$ such that for any Borel
set $E\subset\R^n$  the estimate $\H^1(v(Z_v\cap E))\le
C\|D^nv\|(E)$ holds, where $C$ does not depend on~$E,v$.}
\end{lem}

From this Lemma and from Lemma~\ref{bv-1} we infer easily the

\begin{cor}
\label{corbv2} {\sl  Let $v\in \BV_n(\R^n)$. Then for each
$\varepsilon>0$ there exists $\delta>0$ such that for any Borel
set $E\subset\R^n$  we have that if $\H_\infty^{n-1}(E)<\delta$
then $\H^1(v(Z_v\cap E))\le \varepsilon$. In particular,
$\H^1(v(Z_v\cap E))=0$ whenever $\H^{n-1}(E)=0$.}
\end{cor}

We need a more refined version of Lemma~\ref{lb3} in the $\BV$
case.

\begin{lem}\label{lb3'}{\sl
Suppose $v\in\BV_n(\R^n)$ and $S\subset \R^n$ is an
$(n-1)$-dimensional $\CC^1$-smooth compact manifold (without
boundary). Then there exist $\delta=\delta(S)>0$ such that for any
ball $B=B(z,r)$ with $z\in S$ and $r<\delta$  the estimates
\begin{equation}
\label{bvf1} \sup\limits_{y\in \bar
B_+}|v(y)-P_{B_+,n-1}[v](y)|\le C \|D^nv\|(B_+),
\end{equation}
\begin{equation}
\label{bvf2} \sup\limits_{y\in \bar
B_-}|v(y)-P_{B_-,n-1}[v](y)|\le C \|D^nv\|(B_-)
\end{equation}
hold, where $C$ depends on $n$ only, $B_+,B_-$ are the connected
components of the open set $B\setminus S$, and the polynomials
$P_{B_\pm,n-1}[v]$ are defined by formula~(\ref{0}) with $I$
replaced by~$B_{\pm}$, respectively. Moreover, each function
$v_{B_\pm}(y)=v(y)-P_{B_\pm,n-1}[v](y)$, $y\in B_\pm$, can be
extended from~$\bar B_\pm$ to the whole of $\R^n$ such that
$v_{B_\pm}\in\BV_n(\R^n)$ and
\begin{equation}
\label{0'} \|D^nv_{B_\pm}\|(\R^n)\le C_0 \|D^nv\|(B_\pm),
\end{equation}
 where $C_0$ also depends on $n$ only.
}\end{lem}

The proof of this lemma is similar to the proof of
Lemma~\ref{lb3'} with the following addition: we must apply the
advanced version of Sobolev Extension Theorem from bounded
Lipschitz domains to the whole of $\R^n$ with the estimate of the
norm of the extension operator depending on $n$ and on the
Lipschitz constant of the domain only (see~\cite[Chapter VI, \S
3.2, Theorem~5']{St}\,).

From Lemmas~\ref{lb3'} and  \ref{lb11} (more precisely, from its
proof), we have

\begin{cor}\label{bv5}
{\sl Suppose $v\in\BV_n(\R^n)$ and $S\subset \R^n$ is an
$(n-1)$-dimensional $\CC^1$-smooth compact manifold (without
boundary). Then there exist $\delta=\delta(S)>0$ such that for any
ball $B=B(z,r)$ with $z\in S$ and $r<\delta$  the estimate
\begin{equation}
\label{bvf3} \H^1(v(Z_v\cap B\cap S))\le C \|D^nv\|(B_+),
\end{equation}
holds, where $C$ depends on $n$ only and $B_+$ is a connected
component of the open set $B\setminus S$.}
\end{cor}

The next lemma follows from the elementary observation that for
any finite measure $\mu$ we have that $\mu \bigl( \{ x \in \R^n :
\, 0< \dist (x,S)< \varepsilon \} \bigr) \to 0$ as $\varepsilon
\searrow 0$.

\begin{lem}\label{bv6}{\sl
Suppose $v\in\BV_n(\R^n)$ and $S\subset \R^n$ is an
$(n-1)$-dimensional $\CC^1$-smooth compact manifold (without
boundary). Then for any $\varepsilon>0$ there exists a finite
family of balls $B^j=B(z_j,r_j)$, $j=1,\dots,N$, such that $z_j\in
S$, $r_j<\varepsilon$, and
$$S\subset\bigcup\limits_{j=1}^NB^j,\quad\
\sum\limits_{j=1}^N\|D^nv\|(B^j_+)<\varepsilon.$$}
\end{lem}

Combining these results we find the

\begin{cor}\label{corbv8}{\sl
Suppose $v\in\BV_n(\R^n)$ and $S\subset \R^n$ is an
$(n-1)$-dimensional $\CC^1$-smooth compact manifold. Then
$\H^1(v(Z_v\cap S))=0$.}
\end{cor}

Recall, that a set $K\subset\R^n$ is called $(n-1)$-rectifiable,
if there exists an at most countable family  of $\CC^1$--surfaces
$S_i\subset\R^n$ of dimension~$(n-1)$ such that
$\H^{n-1}\biggl(K\setminus \bigcup\limits_i S_i\biggr)=0.$

We can therefore reformulate the Corollaries~\ref{corbv8},
\ref{corbv2} in the following form.

\begin{cor}\label{corbv9}{\sl
Suppose $v\in\BV_n(\R^n)$ and $K\subset\R^n$ is an
$(n-1)$-rectifiable set. Then $\H^1(v(Z_v\cap K))=0$.}
\end{cor}

The following fact is well-known.

\begin{ttt}[see \cite{Dor}, Theorems~B and 1]\label{Dor2}
{\sl Suppose that $v\in \BV_{n}(\R^n)$. Then there exists a
decomposition $\R^n=K_v\cup G_v$  with the following properties:
\begin{itemize}

\item[(i)] \ $K_v$ is $(n-1)$-rectifiable;

\item[(ii)] \ each $x\in G_v$ is a Lebesgue point for
$\nabla^{n-1}v$, moreover, $\nabla^{k-2}v$ is differentiable at
$x$ in the following integral sense:
\begin{equation}
\label{llc0}\dashint_{B(x,r)}\bigl|\nabla^{n-2}v(y)-\nabla^{n-2}v(x)-\nabla^{n-1}v(x)\cdot(y-x)\bigr|\,dy=o(r)\quad\mbox{
as } r \searrow 0.\end{equation}
\end{itemize}}
\end{ttt}

Now we are able to prove the following main result:

\begin{ttt}\label{MSBV}{\sl
Suppose $v\in\BV_n(\R^n)$. Then $\H^1(v(Z_v))=0$.}
\end{ttt}

\begin{proof}
In view of Corollary \ref{corbv2} and Theorem~\ref{bv''} it is
sufficient to prove the target equality $\H^1(v(Z_v))=0$ only for
a case when $v\in\BV_n(\R^n)\cap\CC^{n-2,1}(\R^n)$, i.e.,
$v\in\CC^{n-2}(\R^n)$ and $\nabla^{n-2}v$ satisfies the Lipschitz
condition
\begin{equation}
\label{llc1} |\nabla^{n-2}v(y)-\nabla^{n-2}v(x)|\le
L|y-x|\quad\mbox{ for all }x,y\in \R^n
\end{equation}
and for some constant $L>0$. Consider the sets $K_v$, $G_v$ from
Theorem~\ref{Dor2}. In view of Corollary~\ref{corbv9} we have
\begin{equation}
\label{fbv2} \H^1(v(Z_v\cap K_v))=0.
\end{equation}
So we need only to prove that $\H^1(v(Z_v\cap G_v))=0$.

Take the decomposition (nondisjoint in general) \ $G_v=G_1\cup
G_2\cup G_3$, where
$$G_1=\{x\in G_v:\exists m=2,\dots,n-2\ \ \nabla^mv(x)\ne0\},$$
$$G_2=\{x\in G_v:\nabla^{n-1}v(x)=0\},$$
$$G_3=\{x\in G_v:\nabla v^{n-2}(x)=0, \ \
\nabla v^{n-1}(x)\ne 0\}.$$ Because of  Corollary~\ref{corbv8} and
the Implicit Function Theorem for smooth functions we have
\begin{equation}
\label{fbv3} \H^1(v(Z_v\cap
G_1))\le\sum_{m=2}^{n-2}\H^1\bigl(v\bigl(\{x\in G_v:\nabla
v(x)=\dots=\nabla^{m-1}v(x)=0,\ \
\nabla^mv(x)\ne0\}\bigr)\bigr)=0.
\end{equation}
On the other hand, by the Coarea Formula (see \cite{EG}) \ $
\|D^nv\|(G_2)=0$, \ hence by Lemma~\ref{bv1},
\begin{equation}
\label{fbv5} \H^1(v(Z_v\cap G_2))=0.
\end{equation}
Now estimate $v(Z_v\cap G_3)$.
 From the integral differentiability (\ref{llc0}) and the Lipschitz
 condition~(\ref{llc1}) it follows that $\nabla^{n-2}v$ is differentiable in
the classical sense for each $x\in G_v$, i.e.,
\begin{equation}
\label{llc3}\forall x\in G_v\quad\bigl|\nabla^{n-2}v(y)-
\nabla^{n-2}v(x)-\nabla^{n-1}v(x)\cdot(y-x)\bigr|=o(r)\quad\mbox{
as } r \searrow 0.\end{equation} Let \ ${\bf e}_i$, \
$i=1,\dots,n$, \ be the unit coordinate vectors of $\R^n$. Denote
$$E_{i,j,k}=\bigl\{x\in G_3:\bigl|\nabla v^{n-2}(x+t{\bf
e}_i)\bigr|\ge\frac1j|t|\ \mbox{ for all }
t\in\bigl[-\frac1k,\frac1k\bigr]\bigr\}.$$ By construction,
$$G_3=\bigcup\limits_{j,k\in\mathbb N,\,i=1,\dots,n}E_{ijk}.$$
It is easy to see (using the Lipschitz condition~(\ref{llc1})\,)
that locally each set $E_{ijk}$ is a graph of some Lipschitz
function of $(n-1)$ variables $(x_1,\dots,\widehat
x_i,\dots,x_n)$, i.e., $E_{ijk}$ is $(n-1)$-rectifiable. Then by
Corollary~\ref{corbv9} $\H^1(v(Z_v\cap E_{ijk}))=0$.
\end{proof}

 \noindent
School of Mathematics, Institute for Advanced Study, Einstein
Drive, Princeton N.J.~08540, USA\\
e-mail: {\it bourgain@ias.edu}
\bigskip

\noindent Sobolev Institute of Mathematics, Acad.~Koptyuga pr., 4,
and Novosibirsk State University, Pirogova Str. 2, 630090
Novosibirsk, Russia\\
e-mail: {\it korob@math.nsc.ru}

\bigskip

\noindent
Mathematical Institute, University of Oxford, 24--29 St.~Giles',
Oxford OX1 4AU, England\\
e-mail: {\it kristens@maths.ox.ac.uk}

\end{document}